\def\PREP{PREP}
\def\ISSAC{ISSAC}
\def\form{PREP}

\ifx\form\PREP
  \documentclass[12pt]{amsart}

\fi

\ifx\form\ISSAC
  \documentclass{sig-alternate}

\fi

\usepackage{amssymb}
\usepackage{amsfonts}
\usepackage{amsthm}
\usepackage{enumerate}
\usepackage{graphicx}
\usepackage{xcolor}
\usepackage{url}
\usepackage{microtype}
\usepackage{verbatim}

\ifx\form\PREP
\usepackage{hyperref}
\definecolor{darkblue}{RGB}{0,0,160}
\hypersetup{
colorlinks,%
citecolor=black,%
filecolor=black,%
linkcolor=darkblue,%
urlcolor=darkblue
}
\fi

\usepackage{tikz}
\usetikzlibrary{patterns}
\tikzstyle box style=[thick, color=black]

\usepackage[utf8]{inputenc}

\newcommand{\excise}[1]{}

\newtheorem{thm}{Theorem}[section]
\newtheorem{lemma}[thm]{Lemma}

\newtheorem{cor}[thm]{Corollary}
\newtheorem{prop}[thm]{Proposition}

\theoremstyle{definition}
\newtheorem{example}[thm]{Example}
\newtheorem{remark}[thm]{Remark}
\newtheorem{defn}[thm]{Definition}

\numberwithin{equation}{section}

\newcommand{\quest}[1]{{\bf(#1)}}

\newcommand{\ring}[1]{\ensuremath{\mathbb{#1}}}

\renewcommand\>{\rangle}
\newcommand\<{\langle}

\newcommand\NN{\ring{N}}

\newcommand\PP{{\ring{P}}}

\newcommand\ZZ{\ring{Z}}

\newcommand\into{\hookrightarrow}

\DeclareMathOperator\width{width}
\DeclareMathOperator\im{im}

\DeclareMathOperator{\Quad}{Quad}

\newcommand{\fourtitwo}{\textsc{4ti2}\ }
\newcommand{\ideal}[1]{\langle #1 \rangle}
\newcommand{\Sym}{S_\infty}
\newcommand{\ZZN}{\bigoplus_\NN\ZZ} 
\newcommand{\ZZNN}{\bigoplus_{\NN\times\NN}\ZZ}
\newcommand{\ZZNk}{\bigoplus_{\NN^k}\ZZ}
\newcommand{\ZZkN}{\bigoplus_{[k-1]\times\NN}\ZZ}
\DeclareMathOperator{\Inc}{Inc}

\begin{document}
\ifx\form\ISSAC
  \conferenceinfo{ISSAC'14}{July 23--25, 2014, Kobe, Japan.}
  \CopyrightYear{2014}
  \crdata{978-1-4503-2501-1/14/07}
\fi

\title{Equivariant lattice generators and Markov bases}

\ifx\form\ISSAC
\numberofauthors{3}
\author{
\alignauthor Thomas Kahle\\
       \affaddr{Otto-von-Guericke Universität, Magdeburg, Germany} \\
       \email{thomas.kahle@ovgu.de}
\alignauthor Robert Krone\\
       \affaddr{Georgia Institute of Technology, Atlanta, USA.}
       \email{krone@math.gatech.edu}
\alignauthor Anton Leykin\\
       \affaddr{Georgia Institute of Technology, Atlanta, USA.}
       \email{leykin@math.gatech.edu}
}
\fi

\ifx\form\PREP
\author{Thomas Kahle}
\address{Otto-von-Guericke Universität\\ Magdeburg, Germany} 
\email{\url{http://www.thomas-kahle.de}}

\author{Robert Krone}
\address{Georgia Institute of Technology\\ Atlanta, USA} 
\email{krone@math.gatech.edu}

\author{Anton Leykin}
\address{Georgia Institute of Technology\\ Atlanta, USA} 
\email{leykin@math.gatech.edu}

\thanks{The second and third author are supported in part by the NSF under grant DMS~1151297.}


\makeatletter
  \@namedef{subjclassname@2010}{\textup{2010} Mathematics Subject Classification}
\makeatother



\fi

\ifx\form\ISSAC
\maketitle
\fi

\begin{abstract}
It has been shown recently that monomial maps in a large class
respecting the action of the infinite symmetric group have, up to
symmetry, finitely generated kernels. We study the simplest nontrivial
family in this class: the maps given by a single monomial.
Considering the corresponding lattice map, we explicitly construct an
equivariant lattice generating set, whose width (the number of
variables necessary to write it down) depends linearly on the width of
the map. This result is sharp and improves dramatically the previously
known upper bound
as it does not depend on the degree of the image monomial.
%
In the case of of width two, we
construct an explicit finite set of binomials generating the toric
ideal up to symmetry. Both width and degree of this generating set are
sharply bounded by linear functions in the exponents of the monomial.
\end{abstract}

\ifx\form\PREP
\maketitle
\fi

\section{Introduction}
\label{sec:introduction}
In algebraic statistics one frequently encounters the question of
stabilization of chains of ideals up to symmetry.  These ideals are
usually parametrized by combinatorial objects (like graphs), and
additionally some numerical parameters.  If the parameters grow, then
the natural symmetry groups acting on the ideals also grow and one
asks if there is stabilization up to symmetry.  On the negative side,
there is the no hope theorem of De Loera and Onn giving large families
with no hope for
stabilization~\cite{loera06:_markov_bases_of_three_way}.  On the
positive side, the independent set theorem of Hillar and Sullivant
~\cite{hillar11:_finit_gr} and a more general recent result on equivariant toric maps~\cite{DEKL13:inf-toric} give large classes with stabilization.
The present paper aims at \emph{constructive} progress towards
understanding the subtle boundary between stabilization and
divergence.

Let $K[Y] := K[y_{ij}\mid i,j\in\NN,\,i\neq j]$ and $K[X] :=
K[x_i\mid i\in \NN]$ be polynomial rings with countably many
indeterminates over a ring~$K$.  Consider the the monomial map
\begin{equation}\label{eq:map-width-2}
\pi:K[Y] \to K[X], \qquad y_{ij} \mapsto x_i^ax_j^b.
\end{equation}
The variables $y_{ii}$ are excluded since they would only contributed
trivial linear relations $y_{ii} = y_{jj}$.  The infinite symmetric
group $\Sym$ of all bijections $\sigma: \NN\to \NN$, acts on~$X$
and~$Y$:
\[
\sigma(x_i) = x_{\sigma(i)} \text{ and } \sigma(y_{ij}) = y_{\sigma(i)\sigma(j)}.
\]
The ring $K[X]$ is {\em equivariantly Noetherian}: every ideal that is
closed under the action of $\Sym$ is finitely generated \emph{up to
symmetry}~\cite{aschenbrenner07:_finit,aschenbrenner09:_erratum}.
Although $K[Y]$ is not $\Sym$-Noetherian, for a large family of
monomial maps $\pi$ (including those considered here) the ideal
$\ker(\pi) \subset K[Y]$---an \emph{infinite-dimensional toric
ideal}---is finitely generated up to symmetry~\cite{DEKL13:inf-toric}.

Any monomial map~$\pi$---finite or infinite-dimensional---is closely
related to its \emph{linearization}~$A_\pi$: the $\ZZ$-linear map on
exponents
\[
A_\pi: \ZZNN \to\ZZN, \qquad A_\pi(e_{ij}) = a e_i + b e_j
\]
where $e_{ij}$ and $e_i$ are the standard basis vectors of $\ZZNN$ and
$\ZZN$, respectively, and $\bigoplus$ denotes the direct sum of
modules.  Each $v\in \ZZNN$ translates to a binomial $y^{v_+} -
y^{v_-} \in K[Y]$ where $(v_\pm)_i = \max\{\pm v_i, 0\}$ are the
positive and negative parts of~$v$, respectively.  The binomials of
this form are exactly those whose terms have greatest common divisor
equal to one.  All minimal generators of toric ideals are of this
form.

The kernel $L = \ker(A_\pi)$ is an infinite-dimensional \emph{(integer)
  lattice}. We call $V\subset \ZZNN$ an \emph{equivariant
  lattice generating set} (or \emph{equivariant lattice generators})
if the $\Sym$-orbits of its elements generate~$L$.  This happens if
and only if the $\Sym$-orbits of $\{y^{v_+} - y^{v_-} \mid v\in V\}$
generate the extension of $\ker(\pi)$ in the ring of {\em Laurent
  polynomials} $K[Y^\pm]$.  An~$\Sym$-generating set of $\ker(\pi)$ is
an {\em equivariant Markov basis} and also spans~$\ker(\pi)K[Y^\pm]$.

\begin{remark}\label{rem:minimal}
Ideally, one would like to define an \emph{equivariant lattice basis},
a generating set whose orbits freely generate the lattice. However,
already in the finite-dimensional case this seems hard: Consider the
sublattice of $\ZZ^3$ generated by $S_3$ acting on $(1,-1,0)$.  One
would like to call $\{(1,-1,0)\}$ an equivariant lattice basis, but
there is a nontrivial linear relation among the elements of the orbit:
\[
(1,0,-1) + (-1,1,0) + (0,-1,1).
\]
Even if there are no relations among elements of the orbit, one has
the problem that lattice bases can not be defined by inclusion
minimality (the integers $2$ and $3$ span $\ZZ$, but no subset does).
One remedy (in the finite-dimensional setting) are matroids over rings
as defined by Fink and Moci~\cite{fink2012matroids}.  There each
subset of the base set is assigned a module, instead of just its rank.
\end{remark}
Any infinite-dimensional object in our setting has {\em truncations}
which are its images in finite-dimensional subspaces.  
\begin{defn}\label{d:truncation}
The {\em $n$-th truncation} of an object indexed by $\NN$ is the
sub-object indexed by $[n] \subset \NN$.  The {\em width} of an
equivariant object is the minimal $n\in\NN$ such that the $n$-th
truncation determines it up to action of~$\Sym$.  The width of an
object which can not be described finitely up to symmetry is infinity.
\end{defn}
The ``objects'' in Definition~\ref{d:truncation} are usually sets of
indeterminates or lattice vectors.
\begin{example}
The $n$-th truncations of the sets of indeterminates $X$ and $Y$ are
$X_n=\{x_i\in X\mid i\leq n\}$ and $Y_n=\{y_{ij}\in Y\mid i,j\leq
n\}$, respectively.  Their widths are $\width(X)=1$,
$\width(Y)=\width(\pi)=\width(\im(\pi))=\width(A_\pi)=2$.
\end{example}
\begin{remark}
An $n$-th truncation of an $\Sym$-invariant object is $S_n$-invariant.
\end{remark}

We study the following problems:
\begin{itemize} \itemsep0em
\item[\quest{L}] Find an {\em equivariant generating set} of~$\ker(A_\pi)$.
\item[\quest{BL}] Bound the {\em width} of an equivariant generating set of smallest width.
\item[\quest{M}] Find an {\em equivariant Markov basis} of~$\ker(\pi)$.
\item[\quest{BM}] Bound the {\em width} of an equivariant Markov basis of smallest width.
\end{itemize}
We are also interested in questions \quest{L}, \quest{BL}, \quest{M},
and \quest{BM} in more general settings than that
in~\eqref{eq:map-width-2}:
\begin{enumerate}
\item[\quest{1}] For $k\in\NN$ let $Y = \{y_\alpha\mid
\alpha\in\NN^k,\, \alpha_j \text{ distinct}\}$ be a set of
indeterminates.  For some fixed values $a_1,\dots, a_k \in\NN$ one can
consider the width~$k$ monomial map
\[
\pi: K[Y] \to K[X], \qquad y_\alpha\mapsto x_{\alpha_1}^{a_1}\cdots x_{\alpha_k}^{a_k}.
\]
\item[\quest{2}] For $k,m,N\in\NN$, let
\begin{align*}
Y &= \{y_{i\alpha} \mid i\in[N],\,\alpha\in\NN^k,\, \alpha_j \text{ distinct}\}\\ 
X &= \left\{x_{lj} \mid l\in [m],\ j\in\NN \right\}
\end{align*}
be sets of indeterminates on which $\Sym$ acts on the second index.
It is known that $K[X]$ is equivariantly Noetherian.  One can now
consider a monomial map $\pi:K[Y]\to K[X]$ of width up to~$k$, that
is, each $y_{i\alpha}$ maps to some monomial in the $x_{lj}$ with
$j\in [k]$.  This is the most general setting for which $\ker(\pi)$ is
known to be finitely generated up to symmetry~\cite{DEKL13:inf-toric}.
\end{enumerate}

A bound in \cite{Hillar13} answers the question \quest{BL1}.  If $\pi$
is defined by a single monomial $x_{\alpha_1}^{a_1}\cdots
x_{\alpha_k}^{a_k}$, there is an equivariant lattice generating set of
$\ker(A_\pi)$ of width~$2d-1$, where $d = a_1+\cdots+a_k$ is the degree
of the image monomial.  In Section~\ref{sec:lattice} we improve this
bound by an explicit construction of equivariant lattice generators
(Theorem~\ref{thm:LatticeBasis}), thus answering
question~\quest{L1}. The width of our basis is two more than the
width~$k$ of the map and thus independent of degree~$d$
(Corollary~\ref{cor:latticeBasisWidth}).

One of our tools is an idea from~\cite{DEKL13:inf-toric}: The map
$\pi$ factors as
\begin{equation}\label{eq:factorize}
\pi: K[Y] \xrightarrow{\phi} K[Z] \xrightarrow{\psi} K[X],
\end{equation}
into two maps that are easier to analyze.  For instance, in the
setting of question~\quest{L1}, we introduce indeterminates $Z_k =
\{z_{ij}, i=1,\dots,k, j \in \NN\}$ and define $\phi,\psi$ by linear
extension of
\[
\phi: y_{\alpha} \mapsto \prod_{i=1}^k z_{i \alpha_i} \qquad \text
{and} \qquad \psi: z_{ij} \mapsto x_j^{a_j}.
\]
The union of a Markov basis for $\ker(\phi)$ and the pullback of a
Markov basis of $\im(\phi) \cap \ker(\psi)$ forms a Markov basis for
$\ker(\pi)$ and similarly for $\ker(\pi)K[Y^\pm]$.

One could hope to compute an equivariant Markov basis of $\ker(\pi)$
by computing a (usual) Markov basis $M_n$ for some $n$-th truncation
$\ker(\pi)_n = \ker(\pi) \cap K[Y_n]$ and check if it
$S_{n+l}$-generates $\ker(\pi)_{n+l}$, for sufficiently many~$l$.
Unfortunately it is unknown how large $l$ needs to be to guarantee
stabilization.

The paper is structured as follows.  In Section~\ref{sec:lattice} we
give an explicit construction of equivariant lattice generators of
$\ker(A_\phi)$ for arbitrary width of the image monomial (in the
setting of \quest{L1}).  Section~\ref{sec:computations} briefly
outlines our computational experiments with truncated Markov bases.
Section~\ref{sec:markov-2} proceeds with an explicit combinatorial
construction of a Markov basis in the width two case.  We conclude in
Section~\ref{sec:questions-problems} with some discussion and open
problems.

\section{\ifx\form\ISSAC\hspace{-.93ex}\fi Equivariant lattice generators}\label{sec:lattice}
We first consider question \quest{L1} in width two ($k=2$).
That is, consider the map $\pi : K[Y] \to K[X]$, defined by $y_{ij}
\mapsto x_i^a x_j^b$.  We look for a generating set of the
extension~$\ker(\pi)K[Y^\pm]$.  To this end we split $\pi$ as
in~\eqref{eq:factorize}, and think of exponents of monomials in
$K[z_{1i},z_{2i}]$ as $2 \times \NN$ matrices.

\begin{prop}\label{prop:2indexLattice}
For $\pi$ with width two, $\ker(\pi)K[Y^\pm]$ is
generated up to symmetry by two binomials:
\[y_{12}y_{34} - y_{14}y_{32}\quad \text{ and }\quad
y_{21}^by_{31}^{a-b} - y_{12}^by^{a-b}_{32}.\]
\end{prop}
\begin{proof}
The \emph{basic quadric} $y_{12}y_{34} - y_{14}y_{32}$ suffices to
generate $\ker(\phi)K[Y^\pm]$ up to symmetry.  This is a classic
result in commutative algebra~(see \cite{room1938geometry} for the
early history) and used often in algebraic statistics: $(2\times
2)$-minors are a Markov basis for the independence model (see
Remark~\ref{rem:indepModel} and
\cite[\S~1.1]{drton09:_lectur_algeb_statis} ).  The non-existing
diagonal variables pose no problem for us since we only need the
Laurent case.

On the exponents of monomials in $K[z_{1i},z_{2i}]$, the linearization
$A_\psi$ of $\psi$ acts by left multiplication with the matrix
$\begin{bmatrix} a & b \end{bmatrix}$.  We can assume that $a$ and $b$ are
relatively prime without loss of generality.  Then $\ker(\psi)$ consists of all matrices of the form
\[ 
\begin{bmatrix} -b \\ a \end{bmatrix}
\begin{bmatrix} n_0 & n_1 &\cdots \end{bmatrix} 
\] 
with $n_i \in \ZZ$.  
Every element in $\im(\phi)\cap \ker(\psi)$ must satisfy $\sum_i n_i =
0$ since $\im(A_\phi)$ is generated by matrices whose two non-zero
entries have the same sign.  The permutations of
\[ 
w := \begin{bmatrix}-b & b & 0 & \cdots \\
a &-a & 0 & \cdots \end{bmatrix} 
\]
form an equivariant lattice generating set for such elements.
Consequently $\im(\phi) \cap \ker(\psi)$ is contained in the lattice
generated up to symmetry by~$w$.  Now
\[ 
\phi(y_{21}^by_{31}^{a-b} - y_{12}^by^{a-b}_{32}) =
z_{13}^{a-b}(z_{12}^b z_{21}^a - z_{11}^b z_{22}^a) 
\] 
which has lattice element $w$, so $y_{21}^by_{31}^{a-b} -
y_{12}^by^{a-b}_{32}$ is the pullback of the generator of $\im(\phi)
\cap \ker(\psi)$.
\end{proof}
 
\begin{remark} 
A generating set in $K[Y]$ for the kernel of a map of width two
requires the 3-cycle cubic $y_{12}y_{23}y_{31} - y_{21}y_{32}y_{13}$
(see Proposition~\ref{prop:phiGens}).  However in the Laurent
ring~$K[Y^\pm]$, this binomial is redundant modulo the basic quadric
$y_{12}y_{34}-y_{14}y_{32}$.  The 3-cycle cubic's exponent is
 \[ 
c = \begin{bmatrix} 0 & 1 &-1 & \\
                       -1 & 0 & 1 & \cdots \\
                        1 &-1 & 0 & \\
                          & \vdots & & \ddots \end{bmatrix}, 
\]
which can be expressed in terms of basic quadrics as
 \[ c = \begin{bmatrix} 0 & 0 & 0 & 0 & \\
                       -1 & 0 & 0 & 1 & \cdots \\
                        1 & 0 & 0 &-1 & \\
                          & \vdots & & & \end{bmatrix}
      + \begin{bmatrix} 0 & 1 & 0 &-1 & \\
                        0 & 0 & 0 & 0 & \cdots \\
                        0 &-1 & 0 & 1 & \\
                          & \vdots & & & \end{bmatrix}
   \ifx\form\ISSAC \]\[ \fi
      + \begin{bmatrix} 0 & 0 &-1 & 1 & \\
                        0 & 0 & 1 &-1 & \cdots \\
                        0 & 0 & 0 & 0 & \\
                          & \vdots & & & \end{bmatrix}. \]
\end{remark}

We now generalize Proposition~\ref{prop:2indexLattice} to
arbitrary~$k$.  When necessary, we write~$\phi^{(k)}$ instead
of~$\phi$ to emphasize the width of the image monomial but usually the
level of generality is clear from the context and we avoid overloading
the notation too much.  Elements of $\ZZNk$ should be thought of
as $k$-dimensional tables of infinite size with integer entries.  Our
setup additionally requires that these tables be zero along their
diagonals (defined as entries indexed by $(i_1,\ldots,i_k)$ with any
$i_j = i_l$ for $j \neq l$).  Let $e_{i_1\ldots i_k}$ denote the
standard basis elements of~$\ZZNk$. Then $Y$ consists of
indeterminates $y_{i_1\ldots i_k} = y^{e_{i_1\ldots i_k}}$.  The
factorization~\eqref{eq:factorize} gives a map $\phi^{(k)}$ as
follows:
\[
\phi^{(k)} : K[Y^\pm] \to K[Z^\pm], \qquad \phi^{(k)} (y_{i_1 \dots i_k}) =
z_{1i_1} z_{2i_2} \cdots z_{k i_k}.
\]

\begin{remark}\label{rem:indepModel}
In algebraic statistics, the \emph{independence model on $k$ factors}
is (the non-negative real part of) the image of the monomial map
$y_{i_1\dots i_k} \mapsto z_{1i_1}z_{2i_2}\cdots z_{k i_k}$ where $i_j
\in [l_j]$ for some
integers~$l_j$~\cite{drton09:_lectur_algeb_statis}.  In algebraic
geometry, this map represents the Segre embedding $\PP^{l_1-1}\times
\cdots \times \PP^{l_k-1} \into \PP^{l_1l_2\cdots l_k - 1}$.  The
coordinate ring of the Segre embedding is presented by quadrics of the
form
\[
y_{i_1\dots i_r \dots i_s\dots i_k}y_{i_1\dots
i'_r\dots i'_s\dots i_k} - y_{i_1 \dots
i'_r \dots i_s \dots i_k}y_{i_1\dots i_r\dots i'_s\dots i_k},
\]
where $i_j,i_j' \in [l_j]$.  This setup differs from ours because
diagonal entries like $y_{11}$ are forbidden for us.  In the analysis
of contingency tables, this restriction is known as a specific
\emph{subtable-sum condition}, namely the sum over all diagonal
entries equals zero~\cite{hara09:_markov}.
Subtable-sum models have more complicated Markov bases than just
independence models, but their lattice bases are still quadratic.
\end{remark}
\begin{prop}\label{p:latticephik}
 The lattice elements
\ifx\form\PREP
 \[
 \Quad^{(k)} := \{ e_{i_1\dots i_r\dots i_s\dots i_k} + e_{i_1\dots
 i'_r\dots i'_s\dots i_k} - e_{i_1\dots i'_r\dots i_s\dots i_k} -
 e_{i_1\dots i_r\dots i'_s\dots i_k}, \mid i_l,i'_l \in [k+2] \}
 \]
\fi
\ifx\form\ISSAC
\begin{align*}
\Quad^{(k)} := \big\{ e_{i_1\dots i_r\dots i_s\dots i_k} & + e_{i_1\dots
i'_r\dots i'_s\dots i_k}\\ - e_{i_1\dots i'_r\dots i_s\dots i_k} & -
e_{i_1\dots i_r\dots i'_s\dots i_k}, \mid i_l,i'_l \in [k+2] \big \}
\end{align*}
\fi 
are an equivariant lattice generating set of~$\ker(A_{\phi^{(k)}})$.
\end{prop}
The elements of $\Quad^{(k)}$ are moves which take two elements
differing in their indices at exactly two positions and then swap the
values in one of those positions.

\begin{proof}[Proof of Proposition~\ref{p:latticephik}]
It is easy to see that $\Quad^{(k)} \subseteq \ker(\phi^{(k)})$.  To
see $\ideal{\Quad^{(k)}} \supseteq \ker(\phi^{(k)})$, we first show
that $\ideal{\Quad^{(k)}}$ contains all elements of the form
\[ 
e_{a_1\ldots a_k} + e_{b_1\ldots b_k} - e_{a_1\ldots a_{k-1},b_k} -
e_{b_1\ldots b_{k-1},a_k}
\]
where $a_1,\ldots,a_k,b_k$ are distinct and also $b_1,\ldots,b_k,a_k$
are distinct.  Now denote $N = \max\{a_1,\ldots,a_k,b_1,\ldots,b_k\}$ and
consider the following telescopic sum in $\Quad^{(k)}$:
\ifx\form\PREP
\small
\begin{align*}
e_{a_1\ldots a_k} - e_{a_1\ldots a_{k-1} b_k} & - (e_{(N+1) a_2\ldots a_k} - e_{(N+1) a_2\ldots a_{k-1} b_k})\\
+ e_{(N+1) a_2\ldots a_k} - e_{(N+1) a_2\ldots a_{k-1} b_k} & - (e_{(N+1) (N+2) a_3\ldots a_k} - e_{(N+1) (N+2) a_3\ldots a_{k-1} b_k})\\
 & \,\,\, \vdots  \\
+ e_{(N+1)\ldots (N+k-2) a_{k-1} a_k} - e_{(N+1)\ldots (N+k-2) a_{k-1} b_k} & - (e_{(N+1)\ldots (N+k-1) a_k} - e_{(N+1)\ldots (N+k-1) b_k})\\
= e_{a_1\ldots a_k} - e_{a_1\ldots a_{k-1} b_k} & - (e_{(N+1)\ldots (N+k-1) a_k} - e_{(N+1)\ldots (N+k-1) b_k}).
\end{align*} 
\fi

\ifx\form\ISSAC
\small
\[\begin{array}{c}
e_{a_1\ldots a_k} - e_{a_1\ldots a_{k-1} b_k} \hfill \\
 \hfill -(e_{(N+1) a_2\ldots a_k} - e_{(N+1) a_2\ldots a_{k-1} b_k})\\
+ e_{(N+1) a_2\ldots a_k} - e_{(N+1) a_2\ldots a_{k-1} b_k} \hfill \\
 \hfill -(e_{(N+1) (N+2) a_3\ldots a_k} - e_{(N+1) (N+2) a_3\ldots a_{k-1} b_k})\\
 \vdots  \\
+ e_{(N+1)\ldots (N+k-2) a_{k-1} a_k} - e_{(N+1)\ldots (N+k-2) a_{k-1} b_k} \qquad \\
 \hfill -(e_{(N+1)\ldots (N+k-1) a_k} - e_{(N+1)\ldots (N+k-1) b_k})\\
= e_{a_1\ldots a_k} - e_{a_1\ldots a_{k-1} b_k} \hfill \\
 \hfill -(e_{(N+1)\ldots (N+k-1) a_k} - e_{(N+1)\ldots (N+k-1) b_k}).
\end{array}\]
\fi

\normalsize
Similarly,
\ifx\form\PREP
$
e_{b_1\ldots b_k} - e_{b_1\ldots b_{k-1} a_k} -
(e_{(N+1)\ldots (N+k-1) b_k} - e_{(N+1)\ldots (N+k-1) a_k}) \in \<\Quad^{(k)}\>
$
\fi
\ifx\form\ISSAC
\begin{align*}
e_{b_1\ldots b_k} - e_{b_1\ldots b_{k-1} a_k} & - \\ (e_{(N+1)\ldots
(N+k-1) b_k} & - e_{(N+1)\ldots (N+k-1) a_k}) \in \<\Quad^{(k)}\>
\end{align*}
\fi
and taking the sum of the two yields the claim.
 
Now let $C := \sum_{I \in \NN^k} c_I e_I \in \ker(\phi^{(k)})$.  For
any $m \in \NN$ let $C_m$ denote the slice of $C$ of entries whose
last index value is $m$, so $C_m := \sum_{I \in \NN^{k-1}}
c_{Im}e_{Im}$.  The sum of the entries of $C_m$ is zero, so $C_m$ can
be decomposed into a sum of terms of the form $ e_{a_1\ldots a_{k-1}
m} - e_{b_1\ldots b_{k-1} m}$.  For each such summand, there is a
corresponding element $e_{a_1\ldots a_{k-1} m} - e_{b_1\ldots b_{k-1}
m} - (e_{a_1\ldots a_{k-1} M} - e_{b_1\ldots b_{k-1} M}) \in
\ideal{\Quad^{(k)}},$ where $M$ is some fixed constant larger than any
index value appearing in~$C$.  Summing up these moves shows
\[ 
C_m - \sum_{I \in \NN^{k-1}} c_{Im}e_{IM} \in \ideal{\Quad^{(k)}}.
\] 
Summing over $m$ shows that $C - D \in \ideal{\Quad^{(k)}}$ where $D
:= \sum_{I \in \NN^k} c_I e_{i_1\ldots i_{k-1} M}$.  Since
$\ideal{\Quad^{(k)}} \subseteq \ker(\phi^{(k)})$, also $D \in
\ker(\phi^{(k)})$.  All non-zero entries of $D$ have $M$ as their last
index entry and dropping it we get an element $D' \in
\ker(\phi^{(k-1)})$.  In the base case $k = 2$, $\phi^{(k-1)}$ is an
isomorphism, so $D'$ and then $D$ are 0 and therefore $C
\in \ideal{\Quad^{(k)}}$.  For $k > 2$, we can assume by induction that
$\ideal{\Quad^{(k-1)}} = \ker(\phi^{(k-1)})$, so $D'$ can be
decomposed into moves in~$\Quad^{(k-1)}$.  Since $D'$ doesn't depend
of the the choice of $M$, we can choose $M$ larger than any index
value used in this decomposition.  Therefore appending $M$ as the $k$-th index value
produces a decomposition of $D$ in $\Quad^{(k)}$, which proves that $C
\in \ideal{\Quad^{(k)}}$.
\end{proof}

To describe $\ker(\pi)K[Y^\pm]$, we proceed to describe $\ker(\psi)$
and its intersection with~$\im(\phi)$, working directly with the
respective linearizations $A_\pi,A_\phi$, and~$A_\psi$.  The
linearization of $\psi : z_{ij} \mapsto x_j^{a_i}$ acts on lattice
elements by left multiplication with the matrix $A_\psi
= \begin{bmatrix} a_1 & \cdots & a_k \end{bmatrix}$. The kernel of
$A_\psi$ is a $(k-1)$-dimensional sublattice of~$\ZZ^k$.  Let $B =
(b_1,\ldots,b_{k-1})$ be a $k\times(k-1)$ matrix whose columns
$b_1,\dots,b_{k-1}$ are a lattice basis of that kernel.  Any element
in $\ker(\psi \circ \phi)$ is homogeneous: the entries of its exponent
vector sum to zero.  Consequently the columns of any $C \in
A_\phi(\ker(A_{\psi \circ \phi)})) = \im(A_\phi) \cap \ker(A_\psi)$
also sum to zero.  With the basis~$B$, if $C = BC'$ with $C' \in
\ZZkN$, then the columns of $C'$ sum to zero as well.  The lattice of
matrices in $\ZZkN$ with zero row sums is generated by the matrices
with a $1$ and $-1$ in any two entries of a particular row, and zero
elsewhere.  Therefore $\im(A_\phi) \cap \ker(A_\psi)$ is contained in
the lattice generated by the orbits of
\[ 
B_i := \begin{bmatrix} b_i & -b_i & 0 & \cdots \end{bmatrix} 
\]
for $1 \leq i < k$.  More specifically $\im(A_\phi) \cap \ker(A_\psi)
\subseteq \ideal{B_1,\ldots,B_{k-1}}_{\Sym} \subseteq \ker(A_\psi)$.
We show constructively that $B_i \in \im(A_\phi)$, so in fact the
orbits of $B_1,\ldots,B_{k-1}$ generate $\im(A_\phi) \cap
\ker(A_\psi)$.  For each $1 \leq j \leq k$ consider the lattice element
\[ 
f_j := e_{a_1 \dots a_{j-1} 1\, a_{j+1} \ldots a_k} - e_{a_1 \dots a_{j-1} 2\, a_{j+1}
\ldots a_k} \in \ZZNk
 \ifx\form\PREP \text{ with $a_l \ge 3$ arbitrary}. \fi
\]
\ifx\form\ISSAC with $a_l \ge 3$ arbitrary. \fi
Applying $A_\phi$, all entries cancel except for the two in the $j$-th
row, producing the matrix with $1$ in the $(j,1)$ entry and $-1$ in
the $(j,2)$ entry.  Any $B_i$ can be expressed as a linear combination
of such matrices.  In particular if $b_i$ has entries $c_1,\ldots,c_k$
then
 \[ w_i := c_1f_1 + \cdots + c_kf_k \in A_\phi^{-1}(B_i). \]
This proves the following theorem.
 
\begin{thm}\label{thm:LatticeBasis}
\ifx\form\ISSAC\hspace{-1.4pt}\fi Up to symmetry,
\ifx\form\ISSAC\hspace{-.8pt}\fi $\Quad^{(k)} \cup
\{w_1,\ldots,w_{k-1}\}$ is an equivariant lattice generating set
of~$\ker(A_{\pi})$, where $k$ is the width of the map~$\pi$.
\end{thm}

\begin{cor}\label{cor:latticeBasisWidth}
The lattice $\ker(A_{\pi})$ has an equivariant lattice generating set
consisting of $(k^2 + k - 2)/2$ elements of width~$k+2$.
\end{cor}

\begin{proof}
Up to $\Sym$-action, each element of $\Quad^{(k)}$ is determined by
the two index positions where the swap takes place.  So $\Quad^{(k)}$
contributes $\binom{k}{2}$ generators.  Additionally we have
$w_1,\ldots,w_{k-1}$, which totals $(k^2 + k - 2)/2$.  Choosing every
$f_j$ with $a_1,\ldots,\hat{a}_j,\ldots,a_k$ being $3,\ldots,k+1$
produces the width bound.
\end{proof}

This generating set is often not minimal in size.  In fact, we can do
away with all of $\Quad^{(k)}$ at the expense of increasing the width
of the~$w_i$.

\begin{cor}\label{cor:latticeBasisWidth2}
The lattice $\ker(A_{\pi})$ has an equivariant lattice generating set
consisting of $k - 1$ elements of width~$2k$.
\end{cor}

\begin{proof}
Suppose $b_l$ is a generator of $\ker(A_{\psi})$ which is non-zero in
the $i$-th coordinate for some $1 \leq i \leq k$.  Choose $w_l$ as in
Corollary~\ref{cor:latticeBasisWidth}, except that $f_i$ is replaced
by
\[ 
f'_i := e_{a'_1 \dots a'_{i-1} 1\, a'_{i+1} \ldots a'_k} - e_{a'_1
\dots a'_{i-1} 2\, a'_{i+1} \ldots a'_k}
\] 
which has $a'_1,\ldots,\hat{a'_i}\,\ldots,a'_k$ equal to
$k+2,\ldots,2k$.  Then for any $j \neq i$ consider the lattice element
$w_l - \sigma w_l$ where $\sigma \in \Sym$ is the permutation
switching $a'_j$ and $2k+1$.  All terms cancel except for $f'_i -
\sigma f'_i$ which (up to permutation) is the element of $\Quad^{(k)}$
which switches the indices at positions $i$ and $j$.

For any generating set $b_1,\ldots,b_{k-1}$ of $\ker(A_{\psi})$, by
Hall's marriage theorem we can assign to each $b_l$ a distinct $i_l$
such that the the $i_l$-th coordinate of $b_l$ is non-zero.  Then
$i_1,\ldots,i_{k-1}$ include all but one of the values from 1 to $k$.
Construct each $w_l$ as above so that it generates the elements of
$\Quad^{(k)}$ corresponding to all pairs $(i_l, j)$ with $j \neq i_l$.
Every pair represented in $\Quad^{(k)}$ includes some $i_l$ so together
$w_1,\ldots,w_{k-1}$ generates all of 
$\Quad^{(k)}$,  Therefore $w_1,\ldots,w_{k-1}$ is a lattice generating set.
\end{proof}

Note that neither the bounds in
Corollary~\ref{cor:latticeBasisWidth}
nor~\ref{cor:latticeBasisWidth2} are
sharp: For example, the kernel of $y_{ij} \mapsto x_i^2x_j$ in
$K[Y^\pm]$ is generated by a single binomial of width three:
$y_{12}y_{32} - y_{21}y_{31}$.

Theorem~\ref{thm:LatticeBasis} settles \quest{L1}.  For \quest{L2}, we
would like to be able to extend these techniques to a more general
domain ring $K[Y^\pm]$ and a more general target ring $K[X^\pm]$.  For
the latter case, where $X = \left\{x_{lj} \mid l\in [m],\ j\in\NN
\right\}$ with $m > 1$ this is straight-forward.  Factoring $\pi$ into
\[ 
\pi: K[Y^\pm] \xrightarrow{\phi} K[Z^\pm] \xrightarrow{\psi} K[X^\pm],
\] 
we have the same $\phi$ as before, with the same kernel.
The linearization $A_\psi$ of $\psi$ is left
multiplication by an $m \times k$ matrix with non-negative
entries.  A lattice basis $\{b_1,\ldots,b_s\}$ for $\ker(A_\psi)$ can be
computed using standard algorithms.  Since any binomials in
$\ker(\pi)$ is homogeneous, and every variable in $Y$ contributes
exactly one to each row in $Z$, the matrices for $\im(A_\phi) \cap
\ker(A_\psi)$ have row sums equal to zero.  Therefore $\im(A_\phi) \cap
\ker(A_\psi)$ is again generated by
\[ 
B_i := \begin{bmatrix} b_i & -b_i & 0 & \cdots \end{bmatrix}.
\]

On the other hand, extending the domain to $K[Y^\pm]$ with $Y =
\{y_{i\alpha} \mid i\in[N],\,\alpha\in\NN^k,\, \alpha_j \text{
distinct}\}$ presents obstacles for $N > 1$.  Here the lattice $Z^\pm$
is represented by $Nk \times \NN$ matrices, with $k$ rows in the image
of each of the $N$ orbits of $Y$.  Our previous argument breaks down
because the matrices corresponding to binomials in $\phi(\ker(\pi))$
need not have all row sums equal to zero, which was critical to the
construction used when $N = 1$.  Binomials in $\ker(\pi)$ need not be
homogeneous, and even homogeneous binomials need not correspond to matrices
in $Z^\pm$ with zero row sums.

\section{Examples and Tools}
\label{sec:computations}

During experimental investigations leading to the results in this
paper we used several different ways to represent binomials in the
various rings.  Let us introduce the most useful ones in the setting
of width two, that is $y_{ij} \mapsto x_i^ax_j^b$.  The extension to
higher width is simple.

A simple way to represent monomials in the various polynomial rings is
with a table of its exponents, as used in the previous section.  These
tables have an infinite number of entries, but only a finite number
are non-zero for a given monomial.

The monomials in $K[Y]$ with $k = 2$ correspond to $\NN \times \NN$
matrices $(a_{ij})_{i,j \in \NN}$ where the entry $a_{ij}$ is the
exponent of $y_{ij}$.  All diagonal entries $a_{ii}$ are zero.  The
action of $\Sym$ simultaneously permutes the rows and columns of the
matrix.  A binomial $y^A - y^B \in K[Y]$ is in $\ker \phi$ if the
corresponding pair of matrices $A$, $B$ have each row sum and each
column sum equal.
The monomials of $K[Z]$ correspond to $2 \times \NN$ matrices and the
action of $\Sym$ permutes the columns.  In particular we are
interested in the monomials in the image of $\phi$, but these are easy
to identify due to \cite[Proposition~3.1]{DEKL13:inf-toric}.  They
correspond precisely to the matrices whose row sums are both equal to
some $d \in \NN$, and whose column sums don't exceed $d$.
Finally, the monomials of $K[X]$ can be represented by infinite row
vectors.  The map $\psi$ corresponds to left multiplication by the $(1
\times 2)$ matrix $[a,b]$.

While thinking about generators we also used the \emph{box shape
formalism} which we explain now.  Clearly every monomial in $K[X]$ is
specified by the exponents of the variables it contains.  An exponent
$a$ in $x_i^a$ can be represented by a column of height $a$ in
position $i$ in some diagram.  For instance the monomial
$x_1^3x_2^2x_3$ displays as
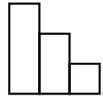
\begin{equation*}
\begin{tikzpicture} [scale=.4]
\draw[style=box style] (0,0) rectangle (1,3);
\draw[style=box style] (1,0) rectangle (2,2);
\draw[style=box style] (2,0) rectangle (3,1);
\end{tikzpicture}.
\end{equation*}
Up to the action of $\Sym$, the order of columns is irrelevant.
Therefore one may choose an arbitrary convention like ordering the
columns by size.
Extending this formalism we represent monomials in the matching
monoid, the image of $K[Y]$ in $K[Z]$, by subdivided columns,
according to the following rule: A variable $z_{ik}$ corresponds to a
box of height $a_i$ in column~$k$.  In the width two case, $z_{1k}$
gives a box of height $a$ and $z_{2k}$ a box of height $b$.  Since we
are using commutative variables, the ordering of boxes in a column
plays no role.  For example, when $a=2, b=1$, the two monomials
$z_{11}z_{13}z_{22}^2$ (the image of $y_{12}y_{32}$) and
$z_{12}z_{13}z_{21}^2$ (the image of $y_{21}y_{31}$) display as
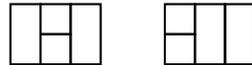
\begin{equation*}
\begin{tikzpicture}[scale=.4]
\draw[style=box style] (0,0) rectangle (1,2);
\draw[style=box style] (1,0) rectangle (2,1);
\draw[style=box style] (1,1) rectangle (2,2);
\draw[style=box style] (2,0) rectangle (3,2);
\end{tikzpicture} \qquad
\begin{tikzpicture}[scale=.4]
\draw[style=box style] (0,0) rectangle (1,1);
\draw[style=box style] (0,1) rectangle (1,2);
\draw[style=box style] (1,0) rectangle (2,2);
\draw[style=box style] (2,0) rectangle (3,2);
\end{tikzpicture} 
\end{equation*}
Note that both monomials have the same image in $K[X]$ which is just
their outer shape, in this case three columns of height two each.
From these displays it is obvious that $\psi(z_{11}z_{13}z_{22}^2 -
z_{12}z_{13}z_{21}^2) = 0$.  Finally, we represent monomials in
$K[Y]$ by decorated box shapes, which also record the information of
which pairs of boxes (one of height $a$, one of height $b$ in
different columns) originated from the same variable.  Consider the
following two decorated box piles:
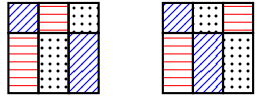
\begin{equation*}
\begin{tikzpicture}[scale=.4]
\filldraw[style=box style, pattern color=red, pattern=horizontal lines] (0,0) rectangle (1,2);
\filldraw[style=box style, pattern color=blue, pattern=north east lines] (0,2) rectangle (1,3);
\filldraw[style=box style, pattern=dots] (1,0) rectangle (2,2);
\filldraw[style=box style, pattern color=red, pattern=horizontal lines] (1,2) rectangle (2,3);
\filldraw[style=box style, pattern color=blue, pattern=north east lines] (2,0) rectangle (3,2);
\filldraw[style=box style, pattern=dots] (2,2) rectangle (3,3);
\end{tikzpicture} \qquad
\begin{tikzpicture}[scale=.4]
\filldraw[style=box style, pattern color=red, pattern=horizontal lines] (0,0) rectangle (1,2);
\filldraw[style=box style, pattern color=blue, pattern=north east lines] (0,2) rectangle (1,3);
\filldraw[style=box style, pattern color=blue, pattern=north east lines] (1,0) rectangle (2,2);
\filldraw[style=box style, pattern=dots] (1,2) rectangle (2,3);
\filldraw[style=box style, pattern=dots] (2,0) rectangle (3,2);
\filldraw[style=box style, pattern color=red, pattern=horizontal lines] (2,2) rectangle (3,3);
\end{tikzpicture} 
\end{equation*}
This display illustrates that $\phi(y_{12}y_{23}y_{31} -
y_{21}y_{32}y_{13}) = 0$. Note that the two monomials are in the same
$\Sym$-orbit, which may or may not happen for binomials in a Markov
basis (see Example~\ref{ex:621} below).

Our computational powers in $K[Y]$ are limited (though not zero).
Therefore it is advantageous for experiments to approximate
equivariant computations with their truncations.  To this end, fix a
truncation width~$n$ and coprime exponents $a>b$.  The computation of
an equivariant Markov basis can be approximated as follows.  Consider
the matrix $A_n$ whose columns are the elements of the orbit of
$(a,b,0,\dots, 0)$ under the action of~$S_n$.  The size of the orbit
is~$n(n-1)$ and each column is indexed by a pair $(i,j)$ of
indices~$i\neq j$.  The group $S_n$ acts on the columns by the rule
$\sigma (i,j) = (\sigma (i), \sigma (j))$.  If $n$ is reasonably
small, say $n=6$, then \fourtitwo \cite{4ti2} computes a usual Markov basis of
this matrix in no time.  A~simple algorithm reduces the result modulo
symmetry: One can check for each element of the usual Markov basis, if
it is in the orbit of some other element, by enumerating the orbit.
The result of this algorithm is the truncated Markov basis.
Computations with these truncated bases have lead us to conjecture the
results of this paper.

\begin{example}\label{ex:621}
On a standard notebook, for $n=6$, $a=2$, $b=1$, the \fourtitwo
computation took only a fraction of a second.  Since Macaulay2 is not
optimized for this kind of computation, determining the representation
of $S_6$ on $S_{30}$, the permutation group of the columns of $A$,
took about two minutes.  The resulting 270 moves, reduced to 5 orbits
within seconds.  Here are their decorated box representations:
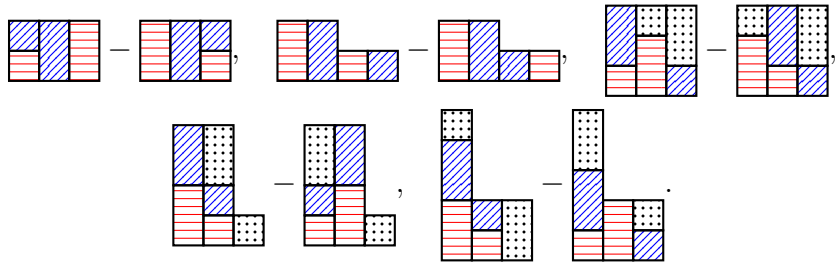
\begin{gather*}
\begin{tikzpicture}[baseline={([yshift=-.5ex]current bounding box.center)}, scale=.4]
\filldraw[style=box style, pattern color=red, pattern=horizontal lines] (0,0) rectangle (1,1);
\filldraw[style=box style, pattern color=blue, pattern=north east lines] (0,1) rectangle (1,2);
\filldraw[style=box style, pattern color=blue, pattern=north east lines] (1,0) rectangle (2,2);
\filldraw[style=box style, pattern color=red, pattern=horizontal lines] (2,0) rectangle (3,2);
\end{tikzpicture}
-
\begin{tikzpicture}[baseline={([yshift=-.5ex]current bounding box.center)}, scale=.4]
\filldraw[style=box style, pattern color=red, pattern=horizontal lines] (0,0) rectangle (1,2);
\filldraw[style=box style, pattern color=blue, pattern=north east lines] (1,0) rectangle (2,2);
\filldraw[style=box style, pattern color=blue, pattern=north east lines] (2,1) rectangle (3,2);
\filldraw[style=box style, pattern color=red, pattern=horizontal lines] (2,0) rectangle (3,1);
\end{tikzpicture}, \quad
\begin{tikzpicture} [baseline={([yshift=-.5ex]current bounding box.center)}, scale=.4]
\filldraw[style=box style, pattern color=red, pattern=horizontal lines] (0,0) rectangle (1,2);
\filldraw[style=box style, pattern color=red, pattern=horizontal lines] (2,0) rectangle (3,1);
\filldraw[style=box style, pattern color=blue, pattern=north east lines] (1,0) rectangle (2,2);
\filldraw[style=box style, pattern color=blue, pattern=north east lines] (3,0) rectangle (4,1);
\end{tikzpicture}
-
\begin{tikzpicture} [baseline={([yshift=-.5ex]current bounding box.center)}, scale=.4]
\filldraw[style=box style, pattern color=red, pattern=horizontal lines] (0,0) rectangle (1,2);
\filldraw[style=box style, pattern color=red, pattern=horizontal lines] (3,0) rectangle (4,1);
\filldraw[style=box style, pattern color=blue, pattern=north east lines] (1,0) rectangle (2,2);
\filldraw[style=box style, pattern color=blue, pattern=north east lines] (2,0) rectangle (3,1);
\end{tikzpicture}, \ifx\form\ISSAC \\ \fi \ifx\form\PREP \quad \fi
\begin{tikzpicture} [baseline={([yshift=-.5ex]current bounding box.center)}, scale=.4]
\filldraw[style=box style, pattern color=red, pattern=horizontal lines] (0,0) rectangle (1,1);
\filldraw[style=box style, pattern color=red, pattern=horizontal lines] (1,0) rectangle (2,2);
\filldraw[style=box style, pattern color=blue, pattern=north east lines] (0,1) rectangle (1,3);
\filldraw[style=box style, pattern color=blue, pattern=north east lines] (2,0) rectangle (3,1);
\filldraw[style=box style, pattern=dots] (1,2) rectangle (2,3);
\filldraw[style=box style, pattern=dots] (2,1) rectangle (3,3);
\end{tikzpicture}
-
\begin{tikzpicture} [baseline={([yshift=-.5ex]current bounding box.center)}, scale=.4]
\filldraw[style=box style, pattern color=red, pattern=horizontal lines] (0,0) rectangle (1,2);
\filldraw[style=box style, pattern color=red, pattern=horizontal lines] (1,0) rectangle (2,1);
\filldraw[style=box style, pattern color=blue, pattern=north east lines] (1,1) rectangle (2,3);
\filldraw[style=box style, pattern color=blue, pattern=north east lines] (2,0) rectangle (3,1);
\filldraw[style=box style, pattern=dots] (0,2) rectangle (1,3);
\filldraw[style=box style, pattern=dots] (2,1) rectangle (3,3);
\end{tikzpicture}, \ifx\form\ISSAC \quad \fi \ifx\form\PREP \\ \fi
\begin{tikzpicture} [baseline={([yshift=-.5ex]current bounding box.center)}, scale=.4]
\filldraw[style=box style, pattern color=red, pattern=horizontal lines] (0,0) rectangle (1,2);
\filldraw[style=box style, pattern color=red, pattern=horizontal lines] (1,0) rectangle (2,1);
\filldraw[style=box style, pattern color=blue, pattern=north east lines] (0,2) rectangle (1,4);
\filldraw[style=box style, pattern color=blue, pattern=north east lines] (1,1) rectangle (2,2);
\filldraw[style=box style, pattern=dots] (1,2) rectangle (2,4);
\filldraw[style=box style, pattern=dots] (2,0) rectangle (3,1);
\end{tikzpicture}
-
\begin{tikzpicture} [baseline={([yshift=-.5ex]current bounding box.center)}, scale=.4]
\filldraw[style=box style, pattern color=red, pattern=horizontal lines] (0,0) rectangle (1,1);
\filldraw[style=box style, pattern color=red, pattern=horizontal lines] (1,0) rectangle (2,2);
\filldraw[style=box style, pattern color=blue, pattern=north east lines] (0,1) rectangle (1,2);
\filldraw[style=box style, pattern color=blue, pattern=north east lines] (1,2) rectangle (2,4);
\filldraw[style=box style, pattern=dots] (0,2) rectangle (1,4);
\filldraw[style=box style, pattern=dots] (2,0) rectangle (3,1);
\end{tikzpicture}, \ifx\form\ISSAC \\ \fi \ifx\form\PREP \quad \fi
\begin{tikzpicture} [baseline={([yshift=-.5ex]current bounding box.center)}, scale=.4]
\filldraw[style=box style, pattern color=red, pattern=horizontal lines] (0,0) rectangle (1,2);
\filldraw[style=box style, pattern color=red, pattern=horizontal lines] (1,0) rectangle (2,1);
\filldraw[style=box style, pattern color=blue, pattern=north east lines] (0,2) rectangle (1,4);
\filldraw[style=box style, pattern color=blue, pattern=north east lines] (1,1) rectangle (2,2);
\filldraw[style=box style, pattern=dots] (0,4) rectangle (1,5);
\filldraw[style=box style, pattern=dots] (2,0) rectangle (3,2);
\end{tikzpicture}
-
\begin{tikzpicture} [baseline={([yshift=-.5ex]current bounding box.center)}, scale=.4]
\filldraw[style=box style, pattern color=red, pattern=horizontal lines] (0,0) rectangle (1,1);
\filldraw[style=box style, pattern color=red, pattern=horizontal lines] (1,0) rectangle (2,2);
\filldraw[style=box style, pattern color=blue, pattern=north east lines] (0,1) rectangle (1,3);
\filldraw[style=box style, pattern color=blue, pattern=north east lines] (2,0) rectangle (3,1);
\filldraw[style=box style, pattern=dots] (0,3) rectangle (1,5);
\filldraw[style=box style, pattern=dots] (2,1) rectangle (3,2);
\end{tikzpicture}.
\end{gather*}
From this computation one conjectures that $\ker(\pi)$ is generated
(up to symmetry) by moves of width and degree at most three.  Note
however that the computation is not a proof, since there is no
a-priori bound on the width.  In principle there could be some hidden
width seven move that our truncated computation has not found.
Theorem~\ref{thm:markov} shows that this is not the case.
\end{example}

The result of the above computation is a small subset of the degree
five equivariant Gröbner basis
in~\cite[Example~7.2]{DEKL13:inf-toric}.  In general, even for width
two, it is unknown whether a finite equivariant Gr\"obner basis exists
(either for $\Sym$ action or that of $\Inc(\NN)$, a monoid of strictly
increasing functions $\NN\to\NN$).

\section{Equivariant Markov Bases \ifx\form\ISSAC\\ \fi(case of width 2)}\label{sec:markov-2}

In this section we return to the width two map from
\eqref{eq:map-width-2} and construct an equivariant Markov basis.  Fix
exponents $a,b\in\NN$ with $\gcd(a,b)=1$.  If the gcd is larger, then the following
results apply after dividing by the gcd since this will not affect the kernel of the map.  The
Markov basis in Theorem~\ref{thm:markov} has two contributions: one
$\phi$-preimage of each element in the two families in
Proposition~\ref{prop:psiGens} and two elements from the following
proposition.

\begin{prop}\label{prop:phiGens}
The 3-cycle cubic $y_{12}y_{23}y_{31} - y_{21}y_{32}y_{13}$ and
the basic quadric $y_{12}y_{34} - y_{14}y_{32}$ are an equivariant
Markov basis of~$\ker(\phi)$.
\end{prop}
The proof of Proposition~\ref{prop:phiGens} is an adaptation of a
standard technique from algebraic statistics.  It appeared in
\cite[Section~5]{aoki05:_markov_monte_carlo} but our version is due to
Jan Draisma and Jan-Willem Knopper.  We give it for the sake of
completeness.
\begin{proof}[Proof of Proposition~\ref{prop:phiGens}]
Representing a variable $y_{ij}$ as a directed edge $i\to j$,
monomials in $K[y_{ij}]$ correspond to finite loop-free directed
multigraphs on~$\NN$.  For each such graph $G$, let $y^G$ denote
corresponding monomial.  A binomial $y^G - y^H \in \ker(\phi)$
corresponds to a pair of graphs with the same in-degree and out-degree
on each vertex.  The proof is by induction on the degree $d$ of the
binomial.
If $G$ and $H$ share an edge, we can divide by that edge and are done
by induction.  If they don't share an edge, then it suffices to find
an applicable 3-cycle cubic or basic quadric to either $G$ or $H$ and
obtain a new graph $G'$ or $H'$ which shares an edge with $H$ or $G$,
respectively.
 
Without loss of generality, let $(1,2)\in G$ be an edge.  Then $H$ has
an edge out from~$1$, which we can assume is~$(1,3)$, and an edge
$(i,2)$ with~$i \neq 1$.  If $i \neq 3$, apply the basic quadric to
the edges $(1,3)$ and $(i,2)$ to get a graph $H'$ with edges $(1,2)$
and~$(i,3)$. Now $G$ and $H'$ share the edge~$(1,2)$.
If $i = 3$, then $G$ has edges $(3,j)$ and $(k,3)$ with $j \neq 2$ and
$k \neq 1$.  If $j \neq 1$ then apply the basic quadric to $(3,j)$
and~$(1,2)$ to get $G'$ with $(3,2)$ and $(1,j)$, sharing $(3,2)$
with~$H$.  Similarly, if $k \neq 2$, apply the basic quadric to
$(k,3)$ and $(1,2)$.  Finally, if $j = 1$ and $k=2$, then $G$ has a
3-cycle $(1,2)$, $(2,3)$, $(3,1)$.  Applying the 3-cycle cubic to
reverse the direction produces $G'$ with $(2,1)$, $(3,2)$, $(1,3)$
which has edges in common with~$H$.
\end{proof}

It remains to find generators for $\im(\phi) \cap \ker(\psi)$.  For
the remainder of this section, we consider the restriction of $\psi$
to $\im(\phi)$---the {\em matching monoid ring}:
\[ \im(\phi) = K[z_{1i}z_{2j} \mid i,j \in \NN,\; i \neq j] \subseteq
K[Z]. \] Because of the interpretation of its generators as matchings
on $[2]\times \NN$, the multiplicative monoid of $\im(\phi)$ is called
the {\em matching monoid}~\cite[Section~2]{DEKL13:inf-toric}.

\begin{prop}\label{prop:psiGens}
As an ideal in the matching monoid ring, $\ker(\psi)$ is generated by
the $\Sym$-orbits of the binomials $z^A - z^B$ from the following two
finite families:
 \begin{enumerate}[(i)]
  \item For each $0 \leq n \leq a-b$,
  \[ A = \begin{bmatrix}b+n & n & c_{13} & c_{14} & \cdots \\
                        0   & a & c_{23} & c_{24} & \cdots \end{bmatrix},
     \ifx\form\PREP \quad \fi
     \ifx\form\ISSAC \]\[ \fi
     B = \begin{bmatrix}n & b+n & c_{13} & c_{14} & \cdots \\
                        a &   0 & c_{23} & c_{24} & \cdots \end{bmatrix}\]
  where $\sum_{j \geq 3} c_{1j} = a-b-n$ and $\sum_{j \geq 3} c_{2j} = n$.
  \item For each $1 \leq n \leq b$,
  \[ A = \begin{bmatrix}b & 0 & a-b+n & 0 & \cdots \\ 
                        0 & a & n     & 0 & \cdots \end{bmatrix},
     \ifx\form\PREP \quad \fi
     \ifx\form\ISSAC \]\[ \fi
     B = \begin{bmatrix}0 & b & a-b+n & 0 & \cdots \\
                        a & 0 & n     & 0 & \cdots \end{bmatrix}.\]
 \end{enumerate}
 Additionally, all these binomials are minimal with respect to
 division in the matching monoid ring.
\end{prop}

A generating set of the kernel $\pi = \psi \circ \phi$ consists of
preimages of the generators in Proposition~\ref{prop:psiGens} in
$K[Y]$, combined with the generators for $\ker(\phi)$ in
Proposition~\ref{prop:phiGens}.  The remainder of this section
comprises the proof of Proposition~\ref{prop:psiGens}.  To deal with
divisibility in the matching monoid, recall that its generators are
the $(2\times \NN)$-matrices that have exactly one entry 1 in each
row, but not in the same column.  In fact, more holds: according to
\cite[Proposition~3.1]{DEKL13:inf-toric}, a monomial $z^A \in K[Z]$ is
contained in the matching monoid if and only if there is some $d$ such
that both row sums of $A$ are equal to $d$ and all column sums of $A$
are $\leq d$ (the matching monoid is normal).  Consequently, a
monomial is divisible by a generator if we can subtract one in two
different columns (reducing the row sum), without violating the new
column bound~$d-1$.

\begin{prop}\label{prop:graver}
As an ideal in the matching monoid ring, $\ker(\psi)$ is generated up
to symmetry by binomials $z^A - z^B$ with
\[ 
A - B = 
\begin{bmatrix}
b & -b & 0 & \cdots \\ 
-a & a & 0 & \cdots
\end{bmatrix}. 
\]
\end{prop}
\begin{proof}
Let $z^A - z^B \in \ker(\psi)$.  Like in Section~\ref{sec:lattice}, the
difference $A - B$ must be of the form
\[ 
\begin{bmatrix}
b \\
-a
\end{bmatrix}
\begin{bmatrix}
n_1 & n_2 & \cdots
\end{bmatrix} 
\]
where the row vector $n = [n_1\;n_2\;\ldots]$ has entries summing to
zero.  Such a vector can be expressed as a sum $n = v_1 + \cdots +
v_s$ where each $v_i$ is in the $\Sym$-orbit of
$\begin{bmatrix}1&-1&0&\ldots\end{bmatrix}$.  Even more, the
decomposition can be chosen \emph{sign-consistently}, that is, each
$v_i$ has 1 in a position $j$ where $n_j >0$ and has $-1$ where~$n_j <
0$.

Consider the sequence $B = B_0, B_1, \ldots, B_s = A$ of matrices in
$\psi^{-1}(B)$ defined by
\[ 
B_i = B + \begin{bmatrix}b \\ -a\end{bmatrix} (v_1 + \cdots + v_i).
\] 
The sequence is monotonic in each entry, and every column sum is also
monotonic.  Note that the all row sums of all $B_i$ are equal to~$d$.
Since $A$ and $B$ are in the matching monoid, they have non-negative
entries and all column sums $\leq d$.  By the monotonicity of the
sequence, each $B_i$ also satisfies these properties and therefore is
also in the matching monoid.  The proof is complete since $z^{B_i} -
z^{B_{i-1}} \in \ker(\psi)$ for any $i$, and
\[ 
B_{i} - B_{i-1} = \begin{bmatrix}b \\ -a\end{bmatrix}v_i =
\sigma_i \begin{bmatrix}b & -b & 0 & \cdots\\ -a & a & 0 &
\cdots\end{bmatrix}
\] for some $\sigma_i \in \Sym$.
\end{proof}

To prove Proposition~\ref{prop:psiGens} we need to intersect the
matching monoid ring with the equivariant ideal generated by binomials
$z^A - z^B$ with
 \[ 
 A - B = \begin{bmatrix}b & -b & 0 & \cdots\\ -a & a & 0 &
 \cdots\end{bmatrix}.
\]
A general such pair $A$, $B$ is of the form
\[ 
A = \begin{bmatrix}c_{11}+b & c_{12}   & c_{13} & c_{14} & \cdots \\
                   c_{21}   & c_{22}+a & c_{23} & c_{24} & \cdots \end{bmatrix}
\ifx\form\PREP \quad \fi
\ifx\form\ISSAC \]\[ \fi
B = \begin{bmatrix}c_{11}   & c_{12}+b & c_{13} & c_{14} & \cdots \\
                   c_{21}+a & c_{22}   & c_{23} & c_{24} &
\cdots \end{bmatrix}. 
\]
Let $C_j = c_{1j} + c_{2j}$ and $R_i = \sum_{j = 1}^\infty c_{ij}$
be the column and row sums, respectively, {\em excluding} the
contributions $a$ and $b$ in the first two columns.

We show that either the pair $(A,B)$ is on the list in
Proposition~\ref{prop:psiGens}, or $A$ and $B$ are both divisible (in
the matching monoid ring) by a common generator.  Let $d = R_1 + b =
R_2 + a$ be the degree of $A$ and $B$ which gives a bound on column
sums: $C_j \leq d-a$ for $j = 1,2$ and $C_j \leq d$ otherwise.  We say
that a column is \emph{loaded} if it achieves its bound.  Loaded
columns are obstacles to dividing by a common factor, since the degree
can't be decreased without also decreasing the loaded columns by the
same amount.  $A$ and $B$ have a common factor if there exist positive
$c_{1j}$ and $c_{2k}$ such that $j \neq k$ and there are no loaded
columns outside of $j$ and $k$.

\begin{proof}[Proof of Proposition~\ref{prop:psiGens}]
We distinguish four cases depending on the locations of the (at most
two) loaded columns.

\noindent
\emph{Case 1: No columns are loaded}.  We have $d > a$, so $R_1$ and
$R_2$ are both positive.  The monomials $z^A$ and $z^B$ have a common
factor if there are positive $c_{1j}$ and $c_{2k}$ in different
columns $j \neq k$, therefore $c_{ij} > 0$ only for one particular
column~$j$. If $j=1$, then $C_1 = R_1 + R_2 = 2d - a - b > d-a$ which
is a contradiction, and similarly for $j=2$.  Consequently $j\geq 3$
and thus $A,B$ are of the second type for some $1 \leq n < b$.

\noindent
\emph{Case 2: Column $j \geq 3$ is loaded}.  Let $C_j = d$.  Since
$\sum_j C_j = 2d - a - b$, any other column has $C_k \leq d - a - b$
and is not loaded.  Because of the bounds $c_{1j} \leq R_1 = d - b$
and $c_{2j} \leq R_2 = d - a$ and the sum $c_{1j} + c_{2j} = d$,
both $c_{1j}$ and $c_{2j}$ are positive.  Again, all other values of
$c$ must be zero or else $A$ and $B$ have a common factor.  Then we
have $d = c_{1j} + c_{2j} = b + c_{1j} = a + c_{2j}$ and thus
$c_{1j} = a$ and $c_{2j} = b$.  Up to symmetry, this is the binomial
of type~2 with $n = b$.
 
\noindent
\emph{Case 3: Exactly one of Columns one and two is loaded}.  Say
column one is loaded.  In this case no column $j$ can be loaded for
$j\geq 3$: If $c_{11} > 0$ then by the divisibility argument $c_{2j} =
0$ for all $j \neq 1$, and similarly if $c_{21} > 0$, then $c_{1j} =
0$ for $j \neq 1$.  Thus either $C_j = 0$ for $j > 1$ or one of $R_1$
or $R_2$ is zero.  The first case leads to a contradiction as in
Case~1.  So all positive $c$ values are in one row, which must be the
first row since $R_1 > R_2$.  This implies $d = a$ and thus that
column one is loaded contradicting the assumption.  By the same
argument, we cannot have column two loaded and column one not loaded.
 
\noindent
\emph{Case 4: Columns one and two are loaded}.  Either $z^A,z^B$
are divisible by a common generator or we are in one of the following
four situations: $c_{11} = c_{12} = 0$; $c_{21} = c_{22} = 0$;
$C_1 = 0$; or $C_2 = 0$.  However since both column 1 and column 2 are
loaded, $C_1 = C_2 = d-a$, so $C_1 = 0$ if and only if $C_2 = 0$, and
these cases are subsumed by the other two.  If $c_{11} = c_{12} =
0$, then $c_{21} = c_{22} = d-a$.  This implies
\[
2(d-a) + \sum_{j\geq 3}c_{2j} = R_2 = d-a 
\]
so $R_2 = 0$.  Therefore we need only consider the case $c_{21} =
c_{22} = 0$.  Here $c_{11} = c_{12} = d-a$ and
\[ 
2(d-a) + \sum_{j \geq 3} c_{1j} = R_1 = d-b. 
\]
Therefore $\sum_{j \geq 3}c_{1j} = a-b-(d-a)$ and $\sum_{j\geq
3}c_{2j} = R_2 = d-a$.  With $n = d-a$ this yields the binomials of
type~1.
\end{proof}

\begin{remark}
We have gone through some of the many bases of an integer lattice
(ideal)~\cite[\S~1.3]{drton09:_lectur_algeb_statis} and consequently
one may ask if it is possible to define an $\Sym$-equivariant Graver
basis.  Graver bases originated in optimization problems in economy
and are now an important tool in the complexity theory of integer
programming \cite[Part~II]{deLoeraHemmecke13}.  Recall that the
\emph{Graver basis} $G$ of a lattice ideal~$I$ in a polynomial ring is
the unique subset of $I$ satisfying two equivalent properties.  First,
$G$ is the set of all {\em primitive} binomials in the ideal, meaning
that for any $x^A - x^B \in I$, there is $x^{A'} - x^{B'} \in G$ such
that $x^{A'}|x^A$ and $x^{B'}|x^B$.  Second, $G$ is the minimal set
such that for every binomial $x^A - x^B \in I$, the difference $A-B$
has a {\em sign-consistent decomposition} using~$G$, meaning that
there is a sequence of exponents $B = B_0, B_1,\ldots,B_s = A$ which
is monotonic in each entry and every $x^{B_{i+1}} - x^{B_i}$ is a
monomial times an element of~$G$.
In a general monoid ring such as the matching monoid ring, these
two properties are not equivalent.  The set of generators in
Proposition~\ref{prop:psiGens} satisfies the sign-consistent
decomposition property when considered as a subring of $K[Z]$ as
demonstrated in the proof of Proposition~\ref{prop:graver}.  However
it fails the primitiveness condition.  A counterexample with $(a,b) =
(2,1)$ is the binomial $z^A - z^B$ with
 \[ 
A = \begin{bmatrix} 3 & 0 & 1 & 0 & 0 & \cdots \\
                        0 & 2 & 0 & 2 & 0 & \cdots \end{bmatrix},
\ifx\form\PREP \quad \fi
\ifx\form\ISSAC \]\[ \fi
B = \begin{bmatrix} 2 & 1 & 0 & 1 & 0 & \cdots \\
                        2 & 0 & 2 & 0 & 0 & \cdots \end{bmatrix}. 
\]
In a monoid ring, the primitiveness condition is the more natural
property to use in the definition of a Graver basis.  The
sign-consistent decomposition property may not be meaningful if the
ring is not a subset of a polynomial ring.  More convincingly, the set
of primitive binomials forms a universal Gröbner basis, a very
important conclusion in the polynomial ring case.
Although our generating set does not contain all primitive binomials,
its $\Sym$-orbits form a universal Gröbner basis
nonetheless.
\end{remark}

\begin{prop}
The $\Sym$-orbits of the generators in Proposition~\ref{prop:psiGens}
form a universal Gröbner basis of $\ker(\psi)$ as an ideal in the
matching monoid ring.
\end{prop}
\begin{proof}
Fix any monomial $z^B$ in the matching monoid ring and a monomial
order $\leq$.  Let $z^A$ be the standard monomial in the equivalence
class of $z^B$ (that the normal form of $z^B$).  From the proof of
Proposition~\ref{prop:graver}, we have a path $B = B_0, B_1,\ldots,B_s
= A$ which is monotonic in each entry and such that each $z^{B_{i+1}}
- z^{B_i}$ is a monomial multiple of an element in $\Sym G$ where $G$
is the generating set in Proposition~\ref{prop:psiGens}.

Suppose this path is not strictly decreasing in the monomial order, so
there is some $z^{B_{i+1}} > z^{B_i}$.  Let $C = A+B_i - B_{i+1}$.
Because of the monotonicity of the sequence, the entries of $C$ are
between $A$ and $B$ so $z^C$ is in the matching monoid, and $z^A >
z^C$.  This contradicts the assumption that $z^A$ is a standard
monomial.  Therefore $\Sym G$ is a Gröbner basis for this order.
\end{proof}

\begin{remark}
In the theory of equivariant Gröbner bases, only monomial orders that
respect the monoid action are considered.  However the set $\Sym G$ is a
Gröbner basis for {\em any} monomial order.
\end{remark}

To get a generating set for $\ker(\pi)$ we combine the results of
Propositions~\ref{prop:phiGens} and~\ref{prop:psiGens}.  In
particular, for each generator $g$ of $\ker(\psi)$, we find a
representative of $\phi^{-1}(g) \subset K[Y]$, and then combine the
resulting list with the two generators of $\ker(\phi)$.
Interestingly, each generator of $\ker(\psi)$ has a {\em unique}
binomial preimage in $K[Y]$.
\begin{thm}
\label{thm:markov}
In the setup of \eqref{eq:map-width-2} with coprime $a>b$,
the following binomials form a Markov basis of~$\ker(\pi)$.
\begin{enumerate}[(i)]
 \item\label{it:i} $y_{12}y_{34} - y_{14}y_{32};$
 \item\label{it:ii} $y_{12}y_{23}y_{31} - y_{21}y_{32}y_{13};$
 \item\label{it:iii} for each $0 \leq n \leq a-b$,
  \[ y_{12}^{b+n} \prod_{j\geq 3} y_{j2}^{c_{1j}} y_{2j}^{c_{2j}} - y_{21}^{b+n} \prod_{j\geq 3} y_{j1}^{c_{1j}} y_{1j}^{c_{2j}} \]
  where $\sum_{j \geq 3} c_{1j} = a-b-n$ and $\sum_{j \geq 3}^\infty c_{2j} = n$;
 \item\label{it:iv} for each $1 \leq n \leq b$,
  \[ y_{12}^{b-n} y_{13}^n y_{32}^{a-b+n} - y_{21}^{b-n} y_{23}^n
  y_{31}^{a-b+n}. \]
\end{enumerate}
The maximum degree of binomials above is $\max(a+b, 2a-b)$ and
\[\width( \ker(\pi) ) = \max(4,a-b+2).\]
\end{thm}
\begin{proof}
The only open items, the upper bound on the degree and the width
formula, are easily checked: first is achieved by generators of
type~(\ref{it:iii})~or~(\ref{it:iv}), second -- by the basic
quadric~(\ref{it:i}) or a generator of type~(\ref{it:iii}).  

To see the sharpness we show that for $n=0$, the two monomials of the
binomial in~(\ref{it:iii}) are the only two elements in their
multidegree.  This multidegree is \[d = (ab, ab, a, a, \dots, a, 0,
\dots)\] where there are $a-b$ entries equal to~$a$.  Let $m\in k[Y]$
be any monomial of multidegree~$d$.  The total degree of $m$
equals~$a$ since $2ab + (a-b)a = a(a+b)$.  Because of the $a$ entries
in $d$, $m$ must divisible by $y_{3j_3}y_{4j_4}\cdots y_{aj_a}$ where
each $j_i$ is either one or two.  Now since the first two entries of
$d$ both equal $ab$, the only possibility is that all $j_i$ are equal.
Consequently the only two monomials of multidegree $d$ are two
monomials in the type~(\ref{it:iii}) binomial for~$n=0$ and whenever
there are only two monomials of a given multidegree, their difference
appears in every Markov basis.
\end{proof}

\begin{remark}\label{r:finite}
As in Proposition~\ref{prop:psiGens} the list of binomials of the
third type in Theorem~\ref{thm:markov} is finite up to $\Sym$-action.
In particular, we need a representative for each partition of the pair
$(a-b-n,n)$ into a sum of pairs of nonnegative numbers such that in no
pair both entries are zero.
\end{remark}

\begin{example}
Reading line-wise from left to right, the decorated box shapes of
moves in Example~\ref{ex:621} are of types~(\ref{it:iii}) (with
$n=0$), (\ref{it:i}), (\ref{it:ii}), (\ref{it:iii}) (with $n=1$), and
(\ref{it:iv})
\end{example}

\begin{remark}\label{r:degree-complexity}
The maximal degree of the generators in Theorem~\ref{thm:markov}
matches the degrees in Table~1 of~\cite{Hillar13}.  However, we stop
short of proving that our generating set is an equivariant Gr\"obner
basis and we doubt that there needs to exist a term order for which it
is one.  According to our experiments in truncations, we expect the
degrees in Gr\"obner bases to exceed those in
Theorem~\ref{thm:markov}.  For instance in width five for $a=2$,
$b=1$, the Markov complexity in Theorem~\ref{thm:markov} is three,
while among many thousand random weight orders we have not found one
with complexity smaller than five.  In fact, we don't even know if
kernels of the form considered here always admit finite equivariant
Gröbner bases.
\end{remark}

\section{Conclusion}
\label{sec:questions-problems}
The main result of \cite{DEKL13:inf-toric} is that once the parameters
of the single monomial map are fixed (e.g., the exponents $a,b$ in the
case of width 2), there exists a finite equivariant Markov basis.
Precisely, the finiteness means that the description of a Markov basis
for a truncated problem does not depend on the width of truncation,
which can be viewed as another parameter ``tending to infinity''.
One striking example of a family of toric maps where Noetherianity of
equivariant kernels fails appears in De Loera and Onn's \emph{no hope
theorem}~\cite{loera06:_markov_bases_of_three_way}: Markov bases for
three-way contingency tables become \emph{arbitrarily complicated}:
the number of different orbits as well as the degrees of their
elements diverge as the sizes of the tables tend to infinity.

\medskip

Our work settles problems \quest{L1} and \quest{M1} in width two via
explicit constructions of equivariant lattice and Markov bases,
respectively.  The sharp width bounds that follow, answering
\quest{BL1} and \quest{BM1} in width two, are linear functions of the
width of the monomial map and of its exponents, respectively.
The more general questions \quest{L2} and \quest{M2} remain open,
and already the combinatorics of \quest{M1} in width larger than two
appears to be a lot more complicated than that in the case of width
two.

Nonetheless, the very modest bounds on the width in the settled cases
shall fuel optimism as to practical computation of up-to-symmetry
generators of kernels of equivariant toric maps in applications.

\bibliographystyle{amsalpha}
\bibliography{eqMarkov}

\end{document}